\documentclass[12pt]{article}
\usepackage{amsmath}
\usepackage{amssymb}
\usepackage{amscd}
\usepackage{amsthm}
\usepackage{amsfonts}
\usepackage{graphicx}
\usepackage{makeidx}         
\usepackage{tabularx} 
\usepackage{hhline}
\usepackage{vmargin} 

\setpapersize[portrait]{A4}
\setmarginsrb{25mm} 
{20mm} 
{20mm} 
{30mm} 
{0mm} 
{0mm} 
{20mm} 
{10mm} 
\newcommand{\BF}{\mathbb F}

\newcommand{\BN}{\mathbb N}

\newcommand{\BQ}{\mathbb Q}
\newcommand{\BR}{\mathbb R}

\newcommand{\BZ}{\mathbb Z}

\newcommand{\ov}{\overline}

   {\setlength{\parskip}{0pt}%
   \begin{enumerate}%
    }%
   {\end{enumerate}}

\newtheorem{Satz}{General Criterion}[section]
\newtheorem{Thm}[Satz]{Theorem}
\newtheorem{CThm}[Satz]{Characterization Theorem}
\newtheorem{RThm}[Satz]{p-adic Representation Theorem}
\newtheorem{Remark}[Satz]{Remark}

\newtheorem{Lemma}[Satz]{Lemma}
\newtheorem{MLemma}[Satz]{Main Lemma}
\newtheorem{Proposition}[Satz]{Proposition}

\begin{document}
\begin{center}
\textbf{\Large{Algebraic Characterization of Rings \\
of Continuous $p$-adic Valued Functions}} \\
\strut \\
Samuel Volkweis Leite\footnote{This paper contains the main result
of the Ph.D. Thesis \cite{L} of the first author written under the
supervision of the second author.} and Alexander Prestel\label{F1} \\
\end{center}
\strut \\

\textbf{Abstract} The aim of this paper is to characterize among the
class of all commutative rings containing $\BQ$ the rings
$C(X,\BQ_p)$ of all continuous $\BQ_p$-valued functions on a compact
space $X$. The characterization is similar to that of M. Stone from
1940 (see \cite{St}) for the case of $\BR$-valued functions. The
Characterization Theorem 4.6 is a consequence of our main result,
the $p$-adic
Representation Theorem 4.5. \\

\section{Introduction}
The ring $C (X, \BR)$ of all $\BR$-valued continuous functions on a
compact space $X$ is an $\BR$-Banach algebra. Not surprisingly there
are numerous characterizations of these rings among the class of all
$\BR$-Banach algebras (see e.g. \cite{A-K}). What is, however,
surprising is M. Stone's purely algebraic characterization of the
rings $C (X, \BR)$ among the class of all commutative rings $A$
containing $\BQ$. The secret of Stone's approach is that he encodes
the space $X$ in a simple algebraic subset $T$ of $A$. Let us
briefly indicate this
approach in modern language. \\

A subset $T$ of a commutative ring $A$ with $\BQ \subseteq A$ is
called a \textit{pre-ordering} of $A$ if it satisfies \vspace{-6pt}
\begin{equation*}
T + T \subseteq T, \quad T \cdot T \subseteq T, \quad a^2 \in T
\textrm{ for all } a \in A, \quad - 1 \not \in T.
\end{equation*}

If the set of sums of squares of $A$ does not contain $-1$, this set
is a pre-ordering of $A$. In case of $A = C (X, \BR)$, the set of
squares already forms a pre-ordering \footnote{Although
pre-orderings on commutative rings have already been used by M.
Stone, the notation ``pre-ordering'' was introduced much later by
Krivine in a systematic study \cite{Kr}.}. The totality of
pre-orderings on $A$ is partially ordered by inclusion and it
carries a natural topology making it a quasi-compact space. The
\textit{real spectrum} of $(A, T_0)$ is the closed set of
pre-orderings $P \supseteq T_0$ satisfying in addition $P \cup - P =
A \textrm{ and } P \cap -P \textrm{ a prime ideal of } A$.

These objects are usually called \textit{orderings} of $A$ (see
\cite{P-D}). The maximal spectrum $X$ of $(A, T_0)$ yields an
isomorphism $A \cong C (X, \BR)$ if $T_0$ satisfies the conditions
required by Stone. Without going into further details let us mention
only the
crucial step in proving this isomorphism. \\

$T_0$ is called \textit{archimedean} if to every $a \in A$ there
exists some $n \in \BN$ such that $n - a \in T_0$. Then the crucial
step is the \textit{Local-Global-Principle}: If $a \in A$ is
strictly positive for all $P \in X$ (i.e. $a \in P \smallsetminus
(-P)$), then $a \in T_0$. In this sense the pre-ordering $T_0$
encodes the space $X$. In case of the polynomial ring $A = \BR [X_1,
\ldots, X_n ]$, for suitable $T_0$ this principle is Schm\"{u}dgen's
famous Positivstellensatz (see
\cite{P-D}, Theorem 5.2.9). \\

In the present paper we treat in a similar way the rings $C(X,
\BQ_p)$ of all $\BQ_p$-valued continuous functions on a compact
space $X$. We end up with a purely algebraic characterization of
these rings among the class of commutative rings $A$ containing
$\BQ$. In order to achieve this, we introduce certain subsets $|$ of
$A \times A$ called $p$-\textit{divisibilities} \footnote{Compared
with the real situation one could as well call them
``pre-$p$-valuations''.}. The totality $D_p (A)$ of
$p$-divisibilities of $A$ is partially ordered by inclusion and
admits a canonical topology making it a quasi-compact space. We call
a $p$-divisibility $|$ a \textit{p-valuation}(-divisibility) if for
all $a, b \in A$ we have \textit{totality}: $a \ | \ b$  or $b \ | \
a$,  and \textit{cancellation}: $0 \nmid c, \ ac \ | \ bc
\Rightarrow a \ | \ b$.

The class of $p$-valuations $|$ extending a given $p$-divisibility
$|_0$ forms a closed subspace Spec $D_p (A, |_0)$, called the
$p$-adic valuation spectrum above $|_0$. Let $X$ denote the maximal
spectrum $\textrm{Spec}^{\textrm{max}} D_p (A, |_0)$. Finally we
call $|_0$ $p$-\textit{archimedean} if for all $a \in A$ there
exists $n \in \BN$ such that $p^{-n} \ | \ a$. The crucial step in
our approach then is the \textit{Local-Global-Principle:}
\begin{center}If $p \ | \ a$ for all $| \in X$, then $p \ |_0 \
a$.\end{center}

This principle is essential for encoding the $p$-adic valuation
spectrum above $|_0$ in the simple algebraic notion of the
$p$-divisibility $|_0$. Compared with the pre-ordering case, the
extension theory of $p$-divisibilities is considerably more
difficult. In case of the integral domain $A= \BQ_p [X_1, \ldots,
X_n]$ the local-global-principle parallels Roquette's profound
result on the ``Kochen-ring'' of $\BQ_p (X_1, \ldots, X_n)$ in
\cite{R}.

\section{Divisibilities on commutative rings}

Let $A$ be a commutative ring with unit $1 \not = 0$. A binary
relation $a \mid b$ on $A$ (in set theoretic terms we shall write
$\mid \subseteq A \times A$) will be called a \textit{divisibility}
on $A$, if for all $a, b, c \in A$ we have \vspace{-6pt}
\begin{enumerate}
\item[(1)] $a \mid a$ \vspace{-6pt} \item[(2)] $a \mid b, b \mid c \Rightarrow a \mid
c$ \vspace{-6pt} \item[(3)] $a \mid b, a \mid c \Rightarrow a \mid
b-c$ \vspace{-6pt} \item[(4)] $a \mid b \Rightarrow ac \mid bc$
\vspace{-6pt} \item[(5)] $0 \nmid 1$.
\end{enumerate}

Easy consequences from these axioms are e.g. $a  \ | \ 0$ and $a \ |
\ -a$. The set $I(\mid) := \{a \in A; 0 \mid a \}$ is a proper ideal
of $A$. For all $\alpha, \beta \in I (\mid)$ and $a, b \in A$ we
have $a \mid b \Rightarrow a + \alpha \mid b+ \beta$.

It follows that \vspace{-6pt}
\begin{equation*}
a + I \mid b + I : \Leftrightarrow a \mid b \vspace{-6pt}
\end{equation*}
defines a divisibility on the quotient ring $\overline{A} = A /
I(|)$. The ideal $I(|)$ will be called the \textit{support} of $|$.

Clearly, if $\delta : A \rightarrow B$ is a homomorphism of
commutative rings with $1$, i.e., $\delta (1) = 1$, and $|$ is a
divisibility on $B$, then \vspace{-6pt}
\begin{equation*}
a |' b : \Leftrightarrow \delta(a) | \delta (b) \vspace{-6pt}
\end{equation*}
defines a divisibility $|'$ on $A$ with support $I(|') = \delta^{-1}
(I(|))$.

We call a divisibility $|$ on $A$ \textit{total} if for all $a, b
\in A$ we have $a | b$ or $b | a$. We shall say that $|$ admits
\textit{cancellation} if for all $c \not \in I(|)$ (i.e., $0 \nmid
c$), $ac | bc$ implies $a|b$. If $|$ is total and admits
cancellation, we shall also call $|$ a \textit{valuation
divisibility}.

\begin{Proposition}
If the divisibility $|$ has cancellation, then $I(|)$ is prime.
\end{Proposition}

\begin{proof}
Assume that $0 | ab$ and $0 \nmid a$. Then cancelling $a$ in $0 \cdot
a | ba$ gives $0 | b$.
\end{proof}

\noindent \textbf{Example 2.2} Let $A$ be an integral domain and $F
= \textrm{ Quot } A$. Then every subring $B$ of $F$ defines a
divisibility $|$ on $A$ by taking \vspace{-12pt}
\begin{equation*}
a | b : \Leftrightarrow a = b = 0 \textrm{ or } (a \not = 0 \textrm{
and } \frac{b}{a} \in B).
\end{equation*}
Note that $|$ clearly has cancellation and $I(|) = \{ 0 \}$.
Conversely, if $|$ is a divisibility with cancellation and $I(|) =
\{ 0 \}$ on $A$, then \vspace{-6pt}
\begin{equation*}
B := \{ \frac{b}{a}; a, b \in A, a | b, a \not = 0 \} \cup \{ 0 \}
\vspace{-6pt}
\end{equation*}
is a subring of $F$.

It is clear that $| \leftrightarrow B$ is a $1 - 1$ correspondence.
Note that $|$ is total if and only if $B$ is a valuation ring of
$F$. Note also that $A$ need not be a subring of $B$. For example
let $A = \BR[X]$ and $B$ the valuation ring of the degree valuation
on $F = \BR(X)$. Then $A \cap B = \BR$. \\

\noindent \textbf{Example 2.3} Let $v:A \rightarrow \Gamma \cup \{
\infty \}$ be a \textit{valuation in the sense of Bourbaki}, i.e.,
$\Gamma$ is an ordered abelian group $I = v^{-1} (\infty)$ is a
prime ideal of $A, \overline{v} : F \rightarrow \Gamma \cup \{
\infty \}$ is an ordinary valuation on the field $F = \textrm{Quot }
\overline{A}$ with $\overline{A} = A / I$, and $v(a) = \overline{v}
(\overline{a})$ for all $a \in A$. Then \vspace{-6pt}
\begin{equation*}
a | b \Leftrightarrow v(a) \leq v(b) \vspace{-6pt}
\end{equation*}
defines a divisibility on $A$ with $I(|) = I$ prime. Clearly $|$ has
cancellation and is total. \\

Our main example here is $A = F = \BQ_p$ and $v_p : \BQ_p
\rightarrow \BZ \cup \{ \infty \}$. We then call \vspace{-6pt}
\begin{equation*}
a |_p b \Leftrightarrow v_p (a) \leq v_p (b) \vspace{-6pt}
\end{equation*}
the \textit{canonical $p$-adic divisibility}. \\

\noindent \textbf{Example 2.4} Let $A = C(X, \BQ_p)$ be the ring of
all continuous functions $f:X \rightarrow \BQ_p$ where $X$ is a
compact space. We call \vspace{-6pt}
\begin{equation*}
f |^* g \Leftrightarrow \forall x \in X (v_p(f(x)) \leq v_p (g(x)))
\vspace{-6pt}
\end{equation*}
the \textit{canonical $p$-adic divisibility} on $A$. If $X$ is
finite and has more than one point, then $|^*$ has no cancellation,
is not total, and $I(|^*) = \{ 0 \}$, but not prime. \\

A valuation $v:F \twoheadrightarrow \Gamma \cup \{ \infty \}$ on a
field $F$ of characteristic $0$ is called a \textit{p-valuation} if
$\Gamma$ is a discretely ordered abelian group with $v(p)$ as
minimal positive element and the residue field $\overline{F}$ of $v$
is the finite field $\BF_p$ of $p$ elements. $(F, v)$ is called
\textit{p-adically closed} if $v:F \twoheadrightarrow \Gamma \cup \{
\infty \}$ is a $p$-valuation, $(F,v)$ is henselian, and the
quotient group $\Gamma / \BZ v(p)$ is divisible. Clearly, $(\BQ_p,
v_p)$ is $p$-adically closed. Every $p$-valued field admits an
algebraic extension that is $p$-adically closed, called a $p$-adic
closure. $p$-adic closures are in general not unique up to
isomorphism. In case $\Gamma = \BZ$, the $p$-adic closure is unique
up to isomorphism as it is the henselization. For more information
the reader is refered to \cite{P-R}.

Returning to a $p$-valued field $(F,v)$ let us simply write $1$ for
the positive minimal value $v(p)$. For every $x \in F$ the quotient
\vspace{-6pt}
$$
\gamma (x) = \frac{1}{p} \cdot \frac{x^p - x}{(x^p - x)^2 -1}
\vspace{-6pt}
$$
is defined and has value $\geq 0$. The operator $\gamma$ is usually
called the \textit{Kochen operator}. It plays in the theory of
$p$-valued fields a similar role as the square operator does in the
theory of pre-ordered fields. \\

\setcounter{Satz}{4} \begin{Thm} Let $F$ be a field of
characteristic $0$ and let $B$ be a subring of $F$ containing the
ring $\BZ [\gamma(F)]$ generated by all $\gamma (x)$ for $x \in F$.
If $p^{-1} \not \in B$, then $B$ is contained in the valuation ring
$O_v$ of some $p$-valuation $v$ on $F$.
\end{Thm}

\begin{proof}
Clearly, $p$ is not a unit of $B$. Thus there exists a prime ideal
$P$ of $B$ with $p \in P$. By Chevalley's place extension theorem
(\cite{E-P}, ch 3.1) there exists a valuation $v$ of $F$ such that
$O_v \supseteq B$ and $M_v \cap B = P$. Since now the valuation ring
$O_v$ containes $\BZ [\gamma (F)]$, but not $p^{-1}$, $v$ is a
$p$-valuation by \cite{P-R}, Lemma 6.1.
\end{proof}

Motivated by this theorem we call a divisibility $|$ on a
commutative ring with $1 \not = 0$, a \textit{p-divisibility} if it
satisfies for all $a,b \in A$ \vspace{-6pt}
\begin{itemize}
\item[(6)] $0 \nmid a \Rightarrow pa \nmid a$, and \vspace{-6pt}
\item[(7)] $p [(a^p b - b^p a)^2 - (b^{p+1})^2] \ | \ [(a^p b - b^p a)
b^{p+1}]$.
\end{itemize}
Note that (6) implies (5). In fact, $0 | 1$ gives $p | 0 | 1$,
contradicting (6).

\begin{Thm}Let $A$ be an integral domain with $\BQ \subseteq A$
and $F =$ \textrm{Quot} $A$ its field of fractions. Then there is a
$1 - 1$ correspondence between $p$-valuation rings $B \subseteq F$
and total $p$-divisibilities $|$ of $A$ that have cancellation and
support $I(|) = \{ 0 \}$.
\end{Thm}

\begin{proof}
If $B \subseteq F = \textrm{Quot } A$ is a $p$-valuation ring, then
Example 2.2 shows that for $a,b \in A$ \vspace{-6pt}
$$
a | b : \Leftrightarrow a = 0 \textrm{ or } \frac{b}{a} \in B
\vspace{-6pt}
$$
gives a total $p$-divisibility on $A$ with cancellation and $I(|) =
\{ 0 \}$.

Conversely, let $| \subseteq A \times A$ be a total $p$-divisibility
with cancellation and $I(|) =  \{ 0 \}$. Then again Example 2.2
together with Theorem 2.5 shows that
$$
B : = \{ \frac{b}{a}; \ a,b \in A, \ a \not = 0, \ a|b \} \cup \{ 0
\}
$$
is a valuation ring of $F$ being contained in the valuation ring
$O_v$ of some $p$-valuation of $F$. It then follows that $B = O_v$.
In fact, the valuation ring $B$ is mapped by the residue map $\delta
: O_v \rightarrow \BF_p$ of $v$ to a valuation ring $\delta(B)$ of
$\BF_p$. As $\BF_p$ is finite, it follows that $\delta(B) = \BF_p$.
Hence also $B = O_v$.
\end{proof}

As we have seen above the canonical $p$-adic valuation $v_p$ defines
on $\BQ_p$ by $ a |_p b \Leftrightarrow v_p (a) \leq v_p(b)$ a total
$p$-divisibility with cancellation and support $\{ 0 \}$. From this
we also see that the canonical $p$-adic divisibility $|^*$ of $C(X,
\BQ_p)$ is in fact a $p$-divisibility. But in general $|^*$
need neither be total nor have cancellation. \\

\section{The divisibility spectrum}

In this section we shall first introduce the divisibility spectrum
of a commutative $A$ with $1 \not = 0$. We then restrict ourself to
the spectrum of $p$-divisibilities assuming that $\BQ \subseteq A$.
This will provide us with some compact (zero-dimensional) space $X$
on which later the elements of $A$ will operate as continuous
functions with
values in $\BQ_p$. \\

Let $A$ be commutative ring with unit $1 \not = 0$. The next theorem
justifies the name `valuation divisibility' in Section 2 for
divisibilities that are total and admit cancellation.

\begin{Thm}
The valuation divisibilities on $A$ correspond $1-1$ to the Bourbaki
valuations of $A$.
\end{Thm}

\begin{proof}
Let first $v: A \rightarrow \Gamma \cup \{ \infty \}$ be a Bourbaki
valuation on $A$, i.e., $I = v^{-1} (\infty)$ is a prime ideal of
$A$, $\overline{v}: \textrm{ Quot } \overline{A} \rightarrow \Gamma
\cup \{ \infty \}$ with $\overline{A} = A/I$ is an ordinary field
valuation, and $\overline{v} (a + I) = v(a)$ for all $a \in A$. Then
for elements $a, b$ from $A$, $ a |^v b \Leftrightarrow v(a) \leq
v(b)$ defines a total divisibility on $A$ having cancellation and
support $I(|^v) = I$.

Conversely, let $|$ be a valuation divisibility on $A$. Then $I
=I(|)$ is prime by Proposition 2.1 and $|$ induces a total
divisibility $\overline{|}$ on the integral domain $\overline{A} =
A/I$ having cancellation and support $\{ 0 \}$. Thus by Example 2.2
the ring \vspace{-6pt}
$$
B := \{ \frac{\overline{b}}{\overline{a}}; \ \overline{a} \not =
\overline{0} \textrm{ and } a|b \} \cup \{ \overline{0} \}
\vspace{-6pt}
$$
is a valuation ring of $F = \textrm{ Quot } \overline{A}$, say $B=
O_{\overline{v}}$ for some ordinary valuation $\overline{v} : F
\rightarrow \Gamma \cup \{ \infty \}$. Now $v(a) :=
\overline{v}(\ov{a})$ clearly defines a Bourbaki valuation on $A$
with $v^{-1}(\infty) = I$ inducing $\overline{v}$ on $\overline{A}$.
By construction of $v$ we have for all $a,b \in A$, $ a | b
\Leftrightarrow v(a) \leq v(b)$. \\

It is obvious that the correspondence between $v$ and $|$ is one to
one. \end{proof}

\begin{Remark}
Assuming $\BQ \subseteq A$ in the construction of Theorem 3.1, all
fields $\textrm{Quot } \ov{A}$, have characteristic $0$. Thus by
Theorem 2.6 the valuation divisibility $|$ of Theorem 3.1 is a
$p$-divisibility if and only if $\ov{v}$ is a $p$-valuation.
\end{Remark}

Now let us introduce \vspace{-6pt}
\begin{equation*}
\begin{array}{lcl}
D(A) & = & \textrm{ class of all divisibilities of } A, \\
D_p(A) & = &\textrm{ class of all } p\textrm{-divisibilities of } A.
\end{array}
\end{equation*}
Note that both classes are closed by taking unions of chains w.r.t.
inclusion. Thus by Zorn's Lemma every ($p$-)divisibility is
contained in some maximal ($p$-)divisibility. On $D=D(A)$ we
introduce the \textit{spectral} topology as the topology generated
by the sets \vspace{-6pt}
$$
U(a,b) = \{ | \in D; \ a \nmid b \} \vspace{-6pt}
$$
where $a,b$ range over $A$. If we add the complements $ V(a,b) = \{
| \in D; \ a|b \}$ to the above generators, we call this finer
topology the \textit{constructible} one.

Identifying a subset $Y$ of $A \times A$ with its characteristic
function and applying Tychonoff's Theorem to the function space $\{
0,1 \}^{A \times A}$ one proves by standard arguments

\begin{Lemma}
The constructible topology on $D(A)$ is compact. Thus the spectral
topology is, in particular, quasi-compact (i.e. every open cover
contains a finite subcover).
\end{Lemma}

We call the class \vspace{-6pt}
$$
\textrm{Spec } D(A) = \{ | \in D(A); \ | \textrm{ is total and
admits cancellation} \} \vspace{-6pt}
$$
the \textit{divisibility spectrum} of $A$, and $ \textrm{Spec } D_p
(A) = D_p (A) \cap \textrm{ Spec } D(A)$
the \textit{p-divisibility spectrum} of $A$. \\

These two classes are as well closed under unions of chains. Thus
again by Zorn's Lemma every element is contained in a corresponding
maximal one. We denote the subclasses of maximal elements by
$$
\textrm{Spec}^{\textrm{max}} D(A) \quad \textrm{ and } \quad
\textrm{Spec}^{\textrm{max}} D_p(A).
$$
\\

\noindent \begin{Thm} \vspace{-12pt} \begin{enumerate}
\item \textit{$D_p(A),$} \textrm{Spec} \textit{$D(A)$, and} \textrm{Spec } \textit{$D_p(A)$ are
closed subclasses of $D(A)$ in the constructible topology, hence are
quasi-compact in both topologies.}
\item $D(A)^{\textrm{max}}$ \textit{and $D_p(A)^{\textrm{max}}$ are
quasi-compact in the spectral topology.}
\item \textrm{Spec} \textit{$D_p (A)$ and} $\textrm{Spec}^{max}
D_p(A)$ \textit{are compact in both topologies, they actually are
$0$-dimensional spaces: $ V (a,b) = U(bp,a)$ for all $a,b \in A$.}
\end{enumerate}
\end{Thm}

\vspace{6pt} \begin{proof} The proofs are straigt forward by
standard arguments. Let us only mention that in 3 one shows that
$V(a,b) = U(bp, a)$ on $\textrm{ Spec } D_p$. In fact by Theorem 3.1
and Remark 3.2 the elements of $\textrm{ Spec } D_p$ correspond to
Bourbaki $p$-valuations. Recall, if $| \in \textrm{ Spec } D_p$,
then there is a $p$-valuation $\ov{v}$ on $\ov{A} = A/I(|)$ such
that $ a | b \Leftrightarrow \ov{v} (\ov{a}) \leq \ov{v}(\ov{b})$.
As $1 = \ov{v}(p)$ is minimal positive, we get $ a | b
\Leftrightarrow pb \nmid a.$
\end{proof}

For a fixed divisibility $|_0$ on $A$ we shall consider the
subclasses of the above introduced classes consisting of extensions
of $|_0$ and denote them by $ D(A, |_0)$ and $D_p (A, |_0)$
respectively. As $D(A, |_0)$ is closed in the spectral topology, all
topological considerations from above remain true for the
relativized classes. \\

In the following the fixed divisibility $|_0$ will always be assumed
to be \textit{p-archimedean}, i.e.,
\begin{enumerate}
\item[(8)] $\forall a \in A \ \exists m \in \BZ : \ p^m |_0 a$.
\end{enumerate}
The canoncial $p$-adic divisibilities on $\BQ_p$ and on $C(X,
\BQ_p)$ both satisfy axiom (8).

\begin{Thm}
Let $A$ be a commutative ring with $\BQ \subseteq A$, and let $|_0$
be a $p$-archimedean $p$-divisibility on $A$. Then an element $|$ of
$\textrm{ Spec } D_p (A, |_0)$ is maximal if and only if $I(|)$ is
prime and the corresponding $p$-valuation $\ov{v}$ on $F = \textrm{
Quot } A/I(|)$ has value group $\BZ$.
\end{Thm}

\begin{proof}
``$\Rightarrow$'' Let $|$ be maximal in $\textrm{ Spec } D_p (A,
|_0)$. By Theorem 3.1 and Remark 3.2 $|$ corresponds uniquely to a
$p$-valuation $\ov{v} : F \twoheadrightarrow \Gamma \cup \{ \infty
\}$. Denoting (as usual) the positive minimal element $\ov{v}(p)$ of
$\Gamma$ by $1, \BZ = \BZ \ov{v} (p)$ is a convex subgroup of
$\Gamma$. Since $|_0$ is archimedean, so is $|$. Hence for every $a
\in A$ there exists some $m \in \BZ$ such that $ m \leq \ov{v}
(\ov{a})$.

If now $\Gamma$ would be bigger than $\BZ$, there existed some $b
\in A \smallsetminus I(|)$ with $m \leq \ov{v} (\ov{b}) \textrm{ for
all } m \in \BZ$. Thus the set $P = \{ \ov{b} \in \ov{A}; m \leq
\ov{v} (\ov{b}) \textrm{ for all } m \in \BZ \}$ forms a non-zero
prime ideal of $\ov{A}$. Taking  $w (\ov{b}+P): = \ov{v} (\ov{b})$
defines a $p$-valuation on the quotient field $F'$ of $\ov{A}/P$
with value group $\BZ$. Setting $a |' b$ in case $w (\ov{a} + P)
\leq w (\ov{b} + P)$ defines a $p$-divisibility $|' \in \textrm{
Spec } D_p (A, |_0)$ strictly containing $|$. This contradicts the
maximality of $|$. Therefore $\Gamma = \BZ$. \\

``$\Leftarrow$'' Now assume that the $p$-valuation $\ov{v}$
corresponding to $|$ has value group $\BZ$ on $F = \textrm{ Quot }
A/I(|)$. If $|' \in \textrm{ Spec } D_p (A, |_0)$ is a proper
extension of $|$ then $I(|) \subsetneqq I(|')$ or, $I(|) = I(|')$
and the valuation ring $O'$ of $\ov{v}'$ properly extends the
valuation ring $O$ of $\ov{v}$. This second case is not possible,
since (by Lemma 2.3.1 of \cite{E-P}) a proper extension $O'$ of $O$
corresponds to a proper convex subgroup of the value group of $O$
which is $\BZ$. Such a subgroup clearly does not exist. In the first
case, choose $a \in I(|') \smallsetminus I(|)$. Since $v$ has value
group $\BZ$ and $\ov{a}$ is non-zero in $A/I(|)$, there exists some
$m \in \BZ$ such that $\ov{v} (\ov{a}) \leq m$, i.e., $a | p^m$. But
then $a \in I(|')$ implies $0 |' a$. Now $| \subseteq |'$ gives $0
|' p^m$, a contradiction.
\end{proof}

So far we did not show that $\textrm{ Spec } D_p (A)$ is non-empty.

\begin{Thm}
Let $A$ be a commutative ring with $\BQ \subseteq A$. Then $\textrm{
Spec } D_p (A)$ is non-empty if and only if there exists a
$p$-divisibility $|$ on $A$. Equivalently, we have that $A$ admits a
ring homomorphism $\delta$ with $\delta (1) = 1$ into some
$p$-valued field. Spec $D_p (A)$ contains a $p$-archimedean element
if and only if $A$ admits a ring homomorphism with $\delta (1) =1$
into the $p$-adic number field $\BQ_p$.
\end{Thm}

\begin{proof}
Assume $\delta : A \rightarrow F$ is a ring homomorphism with
$\delta (1) = 1$ and $(F,v)$ is a $p$-valued field. Then the
definition \vspace{-6pt}
$$
a | b \Leftrightarrow v (\delta (a)) \leq v(\delta (b))
$$
for $a,b \in A$ obviously yields a $p$-divisibility on $A$ with
$I(|) = \textrm{ ker } \delta$. If $(F,v) = (\BQ_p, v_p)$ then
clearly
$|$ is $p$-archimedean. \\

Next let $|'$ be a $p$-divisibility on $A$. By Zorn's Lemma we can
pass to a maximal extension $|$ of $|'$ inside the class of
$p$-divisibilities extending $|'$. Thus also $|$ is a
$p$-divisibility. We want to see that $|$ admits cancellation.  Let
$c \in A$ and assume $0 \nmid c$. We then define $a |^c b$ if $ac \
| \ bc$ for all $a, b \in A$. One easily checks that $|^c$ is a
$p$-divisibility on $A$ extending $|$. As $|$ is maximal, $|=|^c$.
This implies cancellation by $c$. In fact, if $ac \ | \ bc$, then $a
|^c b$ and as $|=|^c$, we get $a \ | \ b$.

Since now $|$ has cancellation, by Proposition 2.1 $I(|)$ is prime
and we may pass to the ring $\ov{A} = A/I(|)$ and its field of
fractions $F = \textrm{ Quot } \ov{A}$. By Example 2.2 the
divisibility $|$ corresponds to the subring
$$
B = \{ \frac{\ov{b}}{\ov{a}}; \ 0 \nmid a, \ a|b \} \cup \{ \ov{0}
\}
$$
of $F$. Since $|$ is a $p$-divisibility, the Kochen relations (7)
imply that $\BZ [ \gamma (F) ]$ is contained in $B$, while (6)
implies that $p^{-1} \not \in B$. Thus by Theorem 2.5 there exists a
$p$-valuation $\ov{v}$ on $F$ such that $B \subseteq O_{\ov{v}}$.
Now by Remark 3.2 the definition \vspace{-6pt}
$$
a |_1 b \Leftrightarrow \ov{v} (\ov{a}) \leq \ov{v} (\ov{b})
\vspace{-6pt}
$$
yields an extension $|_1$ of $|$ that belongs to $\textrm{ Spec }
D_p (A)$. Thus $\textrm{ Spec } D_p(A)$ is non-empty. \\

Finally, let $| \in \textrm{ Spec } D_p(A)$. By Remark 3.2, $|$
induces a $p$-valuation $\ov{v}$ on $\textrm{ Quot } A/I(|)$. Thus
the canonical homomorphism $\delta : A \rightarrow A/I(|)$ maps $A$
to a $p$-valued field. \\

It remains to show that the existence of a $p$-archimedean element
$| \in \textrm{ Spec } D_p(A)$ provides us with some homomorphism
from $A$ to $\BQ_p$.

We may assume that $|$ is maximal in $\textrm{ Spec } D_p(A)$. Then
by Theorem 3.5, $I(|)$ is prime and $|$ corresponds to some
$p$-valuation $\ov{v}$ on $F = \textrm{ Quot } A/I(|)$ with value
group $\BZ$. In that case, however, the completion of $F$ w.r.t.
$\ov{v}$ is isomorphic to the field $\BQ_p$. Thus the desired
homomorphism is just the canonical homomorphism $\delta : A
\rightarrow A/I(|)$.
\end{proof}

\section{$p$-adic representations}

Now let us fix a commutative ring $A$ with $\BQ \subseteq A$
together with a $p$-archimedean $p$-divisibility $|_0$ on $A$. By
Theorem 3.6 the maximal spectrum \vspace{-6pt}
$$
X = \textrm{ Spec}^{max} D_p(A, |_0 ) \vspace{-6pt}
$$
is non-empty, and by Theorem 3.4.(3.) it is a $0$-dimensional
compact space. By Theorem 3.5 every $| \in X$ induces a canonical
homomorphism
$$
\alpha_| : A \rightarrow \ov{A} = A/I(|) \subseteq F := \textrm{
Quot } \ov{A}
$$
together with a $p$-valuation $\ov{v} : F \rightarrow \BZ \cup \{
\infty \}$ such that $ a | b \Leftrightarrow \ov{v} (\ov{a}) \leq
\ov{v} (\ov{b})$ for all $a, b \in A$. The completion of $F$ w.r.t.
$\ov{v}$ is just the field $\BQ_p$ of $p$-adic numbers with $\ov{v}$
being the restriction of $v_p$ to F.\footnote{Note that every
element of $F$ has a canonical expansion as a power series in the
uniformizer $p$ with coefficients from $\{ 0, 1, \ldots, p-1 \}$
(cf. \cite{E-P}, Proposition 1.3.5).} As $\BQ$ is dense in $\BQ_p$
w.r.t. the topology induced by the $p$-adic valuation $v_p$ on
$\BQ_p$, the embedding of $F$ into $\BQ_p$, is uniquely determined.
Thus every $| \in X$ yields a canonical homomorphism
$$
\alpha_| : A \rightarrow \BQ_p
$$
with $ a | b \Leftrightarrow v_p (\alpha_| (a)) \leq v_p (\alpha_|
(b))$ for all $a,b \in A$. Therefore, every $a \in A$ induces a
canonical map $\widehat{a}$ from $A$ to $\BQ_p$ by taking
\vspace{-6pt}
$$
\widehat{a} (|) := \alpha_| (a) \vspace{-6pt}
$$
for every `point' $|$ in $X$. \\

\begin{Thm}
Let $A$ be a commutative ring with $\BQ \subseteq A$ and let $|_0$
be a $p$-archimedean $p$-divisibility on $A$. Then the map
$\widehat{a}$ is continuous for every $a \in A$. Therefore $\phi : A
\rightarrow C(X, \BQ_p )$ defined by $\phi (a) = \widehat{a}$ is a
homomorphism of rings with dense image $\phi (A)$ in $C(X, \BQ_p )$,
satisfying
$$
a |_0 b \Rightarrow \phi (a) |^* \phi (b), \textrm{ for all } a, b
\in A.
$$
\end{Thm}

\begin{proof}
As $\BQ$ is dense in $\BQ_p$, the sets
$$
U_n (r) = \{ x \in \BQ_p ; \ v_p (x - r) \geq n \}, \ r \in \BQ, \ n
\in \BN
$$
form a base for the topology on $\BQ_p$. Thus it suffices to show
that the preimage of $U_n (r)$ under $\widehat{a}$ is open in the
topology of $X$. This, however, follows from Theorem 3.4.(3.) and
the fact that
$$
(\widehat{a})^{-1} (U_n (r)) = \{ | \in X; \ p^n | a - r \} = V
(p^n, a-r) \cap X
$$
for all $a \in A, r \in \BQ$ and $n \in \BN$.

In order to show that $\phi (A)$ is dense in $C(X, \BQ_p)$ w.r.t.
the maximum norm it suffices by the $p$-adic Stone-Weierstrass
Approximation (see \cite{K}) to show that two different points of
$X$, say $|_1 \not = |_2$ can always be separated by some function
$\widehat{a}$, i.e., $\widehat{a} (|_1) \not = \widehat{a} (|_2)$:
Let $a, b \in A$ distinguish $|_1$ from $|_2$, say $a |_1 b$ and $a
\nmid_2 b$. Then either $\widehat{a}$ or $\widehat{b}$ separates
$|_1$ from $|_2$, as it is easily checked.
\end{proof}

Let $X$ be a compact space. We then denoted by $C(X, \BQ_p)$ the
ring of all $\BQ_p$-valued continuous functions on $X$. This ring
carries a canonical $p$-adic norm which makes it a $p$-adic Banach
algebra over $\BQ_p$. The norm is defined by \vspace{-6pt}
$$
\parallel f \parallel^* := \textrm{ max } \{ | f(x) |_p; \ x \in X \} \vspace{-6pt}
$$
where $| \ |_p$ is the $p$-adic absolute value on $\BQ_p$ defined by
$|x|_p = p^{-v_p (x)}$.

The norm $\parallel \ \parallel^*$ on $C(X, \BQ_p)$ is even
\textit{power multiplicative}, i.e., for all $n \in \BN$
\vspace{-6pt}
$$
\parallel f^n \parallel^* = ( \parallel f \parallel^*)^n. \vspace{-6pt}
$$

Theorem 4.1 provides us with a homomorphism $\phi : A \rightarrow C
(X, \BQ_p)$ with dense image. We have, however, no information about
the kernel of $\phi$. In order to achieve this goal we shall
introduce one more condition on the $p$-divisibility $|_0$ of
Theorem 4.1.

Let us assume that $|$ is a $p$-archimedean $p$-divisibility on the
commutative ring $A$ with $\BQ \subseteq A$. We can then define for
every $a \in A$ \vspace{-6pt}
$$
\textrm{ord } a := \textrm{ sup } \{ m \in \BZ; \ p^m | a \} \in \BZ
\cup \{ \infty \} \textrm{ and } \parallel a \parallel = p^{-
\textrm{ ord } a}.
$$

\noindent \begin{Lemma} \textit{For all $a,b \in A, r \in \BQ$ we
get}
\begin{enumerate}
\item[\textit{(a)}] $\parallel a+b \parallel \leq$ \textrm{max}
$(\parallel a \parallel, \parallel b \parallel)$
\item[\textit{(b)}] $\parallel a \cdot b \parallel \leq \parallel a \parallel \ \parallel b \parallel$
\item[\textit{(c)}] $\parallel r \parallel = | r |_p$
\item[\textit{(d)}] $\parallel ra \parallel = | r |_p \parallel a
\parallel$.
\end{enumerate}
\end{Lemma}

\begin{proof}
(a) and (b) are easily checked. (c) is equivalent to $\textrm{ord }
r = v_p (r)$, and will be shown in Proposition 4.3 below.
\newline (d) then follows from $\parallel ra \parallel \leq
\parallel r \parallel \ \parallel a \parallel = |r|_p \ \parallel a
\parallel$ and
$\parallel a \parallel = \parallel r^{-1} r a \parallel \leq
\parallel r^{-1}\parallel \ \parallel ra \parallel.$
In fact, since by (c), $\parallel \ \parallel$ is multiplicative on
$\BQ$, we then get $ | r |_p \ \parallel a \parallel = \parallel
r^{-1} \parallel^{-1} \
\parallel a \parallel \leq \newline \parallel ra \parallel.$
\end{proof}

It remains to show $\textrm{ord } r = v_p (r)$ for $r \in \BQ$. This
follows from
\begin{Proposition}
The only p-archimedean divisibility with $p \nmid 1$ of the field
$\BQ$ of rational numbers is the one obtained by the p-adic
valuation $v_p$.
\end{Proposition}

\begin{proof}
The support $I(|)$ is a proper ideal of $\BQ$. Hence $I(|) = \{ 0
\}$. Moreover, as $\BQ$ is a field, axiom (4) implies that $|$ has
cancellation. Thus by Example 2.2 it suffices to show that the ring
$B = \{ \frac{b}{a}; \ a, b \in \BQ, \ a \not = 0, \ a|b \} \cup \{
0 \}$ contains the valuation ring $\BZ_{(p)}$ of $v_p$ restricted to
$\BQ$. In fact, then also $B$ is a valuation ring of $\BQ$, hence
has to be equal to $\BZ_{(p)}$ (cf. \cite{E-P}, Theorem 2.1.4). Note
that $B \not = \BQ$ as $p^{-1} \not \in B$.

Let $n,m \in \BZ$ and $n$ prime to $p$. We have to show that $n |
m$. As $|$ is $p$-archimedean there exists $r \in \BN$ such that $
p^{-r} | n^{-1}$. Therefore $n | p^r$. Since $n$ is prime to $p$
there exist $k, l \in \BZ$ with \vspace{-6pt}
$$
k p^r + ln =1.
$$
Since $n | p^r$, also $n | k p^r$. Clearly also $n | ln$. Thus (by
(3)) $n | 1$. Hence $n | m$.
\end{proof}

By Lemma 4.2, $\parallel \ \parallel$ is a sub-multiplicative
$p$-adic semi-norm on $A$. In the next lemma we shall give
equivalent conditions for $\parallel \ \parallel$ to be even power
multiplicative. Note that in this case $\parallel a^n \parallel = 0$
is equivalent to $\parallel a \parallel = 0$. It is well-known that
power multiplicativity is already implied from the case $n = 2$.

\begin{MLemma} Let $|_0$ be a $p$-archimedean $p$-divisibility on
$A$. Then the following three conditions are equivalent:
\vspace{-6pt}
\begin{itemize} \item[(i)] $p |_0 a^2 \Rightarrow p |_0 a$ for all
$a \in A$, \item[(ii)] the norm $\parallel \
\parallel$ defined by $|_0$ is power multiplicative. \item[(iii)]
(Local-Global-Priciple) Let $X = \textrm{ Spec}^{\textrm{max}} D_p
(A, |_0)$. Then $p \ | \ a$ for all $| \in X$ implies $p |_0 a$.
\end{itemize}
\end{MLemma}

\begin{proof}
(iii) $\Rightarrow$ (i) follows from Theorem 4.1 and the fact that
all $|\in X$ satisfy (i). \newline (i) $\Rightarrow$ (ii): As $\parallel
a^2
\parallel \leq
\parallel a
\parallel^2$ is obvious, it remains to prove $\parallel a \parallel^2 \leq \parallel a^2
\parallel$. By the definition of $\parallel \ \parallel$, this
amounts to prove that ord $a^2 \leq 2 \textrm{ ord }a$. Let $m =
\textrm{ ord }a$ and assume $p^{2m+1} \ | a^2$. Then clearly $p |
(ap^{-m})^2$. Hence by (i) we would get $p | ap^{-m}$ or
equivalently $p^{m+1} | a$, a contradiction. \newline (ii)
$\Rightarrow$ (iii): Let us assume $p \nmid_0 a$. We shall then
construct some extension $| \in \textrm{ Spec } D_p (A, |_0)$ such
that $a \ | \ 1$. This clearly implies $p \nmid a$. The extension
$|$ of $|_0$ will be obtained in three steps: \begin{itemize}
\item In step 1 we construct $|_1 \supseteq |_0$ such that $a \ |_1 \
1$ and $|_1$ satisfies all axioms of a $p$-divisibility except (6).
Instead, we shall only obtain $p \nmid 1$. \item In step 2 we
maximalize $|_1$ to $|_2$ such that $|_2$ satisfies axiom (6), hence
is a $p$-divisibility.
\item In step 3 we apply Theorem 3.6 to $D_p (A, |_2)$ in order to obtain
$| \in \textrm{ Spec } D_p (A, |_0)$ with $a \ | \ 1$. \end{itemize}

For step 1 and 2 we need a little preparation: We call an additive
subgroup $C$ of $A$ \textit{convex} w.r.t. $|$ if for all $a, b \in
A$ we have: $a \in C, a \ | \ b \Rightarrow b \in C$. For a subset
$S$ of $A$ we define the convex group $C(S)$ generated by $S$ to be
obtained by iterating countably many times in alternating order the
two operations
$$
\begin{array}{ccl} G(S) & = & \textrm{ additive group generated by } S \\
M(S) & = & \{ b \in A; x \ | \ b \textrm{ for some } x \in S \}.
\end{array}
$$
Then $C(S)$ is a convex subgroup of $A$ containing $S$. The operator $C$ obviously
satisfies $S \subseteq C (S) = CC (S)$ and $a C (S) \subseteq C (aS)$. Moreover we have
$$
a \ | \ b \Rightarrow C ( \{ b \} \cup S) \subseteq C ( \{ a \} \cup
S).
$$

\textbf{Step 1:} We define $x |_1 y : \Leftrightarrow y a^r \in C (
\{ x a^i; 0 \leq i \leq r \} )$ for some $r \in \BN$. First observe
that $|_1$ extends $|_0$. In fact: $x |_0 y \Rightarrow y \in C (x)
( r = 0)$. Moreover we get $a \ |_1 \ 1$ since $a \in C ( \{ a, a^2
\} ) (r = 1)$. Next one checks the axioms (1) - (4) using the above
mentioned properties of the operator $C$. The axioms (7) and (8)
follow from $|_0 \subseteq |_1$. It remains to prove $ p \nmid_1 1$
(then also axiom (5) follows). Let us assume on the contrary the
existence of some $r \in \BN$ such that
$$
a^r \in C ( \{ p a^i; 0 \leq i \leq r \}).
$$
By (ii) we have ord $a^i = i$ ord $a$. From our assumption $p
\nmid_0 a$ we get ord $a \leq 0$. Hence we have the following
contradiction:
$$
\textrm{ord } a^r \geq 1 + \textrm{ min }_{0 \leq i \leq r} \textrm{
ord } a^i = 1 + \textrm{ ord } a^r.
$$

\textbf{Step 2:} Let now $|_2$ be an extention of $|_1$ maximal with
the properties (1) - (4), (7), (8) and $p \nmid_2 1$. We then prove
(6) for $|_2$. Assume $cp |_2 c$ for some $c \in A$. Since $|_2$ is
$p$-archimedean, to every $b \in A$ we find some $m \in \BN$ such
that $1 \ |_2 \ p^m b c$. Applying $pc \ |_2 c$ iteratively yields
$p \ |_2 \ p^{1+m} b c \ |_2 \ p^m bc \ |_2 \ p^{m-1} b c \ |_2
\cdots |_2 \ pbc \ |_2 \ bc.$

Now define
$$
x \ |' \ y : \Leftrightarrow y \in C (cA \cup \{ x \} ).
$$
Clearly $|'$ extends $|_2$. The axioms (1) to (4) are easy to check
and the axioms (7) and (8) are interited. Since $p\ |_2 \ bc$ for
all $b \in A$, we find $C (cA \cup \{ p \} ) \subseteq C (p)$. As $p
\nmid_2 1$ we therefore get $1 \not \in C (cA \cup \{ p \} )$, i.e.
$p \nmid ' 1$. Since $|_2$ was maximal with these properties, we
have $|_2 = | '$, and therefore $0 \ | ' \ c$ yields $0 \ |_2 \ c$.
Thus we have shown that $|_2$ is a $p$-divisibility.
\end{proof}

\newpage \noindent \begin{RThm} \textit{Let $A$ be a commutative ring with
$\BQ \subseteq A$ that admits a $p$-archimedean $p$-divisibility
$|_0$ satisfying $p |_0 a^2 \Rightarrow p |_0 a$ for all $a \in A$.
Then the homomorphism $\phi: A \rightarrow C(X, \BQ_p)$ of Theorem
4.1 with} $X = \textrm{ Spec}^{max} D_p (A, |_0)$ \textit{satisfies}
$\parallel \phi (a)
\parallel^* = \parallel a \parallel_0.$ \textit{Consequently:
\begin{itemize}
\item ker $\phi = \{ a \in A; p^n\ |_0 \ a \textit{ for all } n \in \BN \}$,
\item $\phi$ is injective, if the semi-norm $\parallel \ \parallel_0$
defined by $|_0$ is a norm,
\item $\phi$ is surjective, if $A$ is complete w.r.t. the norm $\parallel \
\parallel_0$.
\end{itemize}}
\end{RThm}

\begin{proof}
By Theorem 4.1 and Main Lemma 4.4 we have $\parallel \phi (a)
\parallel^* =
\parallel a
\parallel_0$. Thus if $\parallel \ \parallel_0$ is a norm, $\phi (a) =
0$ implies $\parallel a \parallel_0 = 0$ and hence $a = 0$. This
proves injectivity.

In order to get surjectivity, let $f \in C(X, \BQ_p)$ be given. As
$\phi (A)$ is dense in $C(X, \BQ_p)$ by Theorem 4.1, there exists a
sequence $(\phi (a_n))_{n \in \BN}, a_n \in A$, converging to $f$.
Then clearly $(a_n )_{n \in \BN}$ is a Cauchy-sequence in $(A,
\parallel \ \parallel_0)$. Thus by completeness there exists a limit
$a$ of $(a_n)_{n \in \BN}$ in $A$. Now $\phi (a) = f$.
\end{proof}

From Theorem 4.5 we finally get our
\begin{CThm}
Let $A$ be a commutative ring with $\BQ \subseteq A$. Then, as a
ring, $A \cong C(X, \BQ_p)$ for some compact (actually
$0$-dimensional) space $X$ if and only if there exists a
$p$-divisibility $a \mid b$ on $A$ such that \vspace{-6pt}
\begin{enumerate}
\item[(i)] $A$ is $p$-archimedean with respect to $\mid$,\vspace{-6pt}
\item[(ii)] the $p$-adic semi-norm canonically
defined by $\mid$ on $A$ is a norm satisfying $\parallel a^2
\parallel = \parallel a \parallel^2$ for all $a \in A$,\vspace{-6pt}
\item[(iii)] $A$ is complete with respect to this norm.
\end{enumerate}
\end{CThm}

So far the characterization 4.6 does not seem to be a completely
algebraic one, as it involves the binary relation $|$. There is,
however, a way to avoid this. The canonical $p$-adic divisibility
$|^*$ on $C(X, \BQ_p)$ can actually be algebraically expressed in
the following way

\begin{Proposition}
The canonical $p$-adic divisibility $|^*$ on $C(X, \BQ_p)$, $X$ a
compact space, satisfies for all $f,g \in C(X, \BQ_p)$ \vspace{-6pt}
$$
g |^* f \Leftrightarrow \exists h \ h^q = g^q + p f^q \vspace{-6pt}
$$
where $q \in \BN$ is a prime different from $p$.
\end{Proposition}

\begin{proof}
``$\Leftarrow$'' Let $x \in X$. Then the values of $g^q (x)$ and $p
f^q (x)$ are different. From $h^q = g^q + p f^q$ we see that the
value of \vspace{-6pt}
$$
(g^q + p f^q) (x)
$$
has to be divisible by $q$. Hence $v_p (g^q (x)) < v_p (p f^q (x))$
which clearly implies $v_p (g (x)) \leq v_p (f (x))$. Thus by
definition $g |^* f.$ \\
``$\Rightarrow$'' Assuming $v_p (g(x)) \leq v_p (f (x))$  for all $x
\in X$ we have to construct a continuous function $h : X \rightarrow
\BQ_p$ such that $h^q = g^q + p f^q$.

Using the fact that the function $g^q + p f^q$ can only take values
in $\BZ$ all of which are divisible by $q$, the fact that the
residue class field is finite, and by patching $h$ from suitable
continuous functions, we are reduced to the case where $g = 1$ on an
open and closed subset $Y$ of $X$. Now we can apply Hensel's Lemma
to the 1-unit $1 + p f (x)^q$ (as the characteristic of the residue
field is different to $q$).
\end{proof}

Using Proposition 4.7 we may replace any use of $a | b$ in the
Characterization Theorem 4.6 by the algebraic expression
\vspace{-6pt}
$$
\textrm{(*)} \quad \exists c \ c^q = a^q + p b^q,
$$
requiring in addition that (*) is a $p$-divisibility satisfying
(i)-(iii). This way we obtain a completely algebraic
characterization of the rings $C(X, \BQ_p)$ with $X$ compact.

\strut \\

\strut \\
\strut \\
\strut \\
\strut \\

samuelvolkweisleite@gmail.com \newline alex.prestel@uni-konstanz.de

\end{document}